\documentclass[letterpaper,11pt,twoside,keywordsasfootnote,addressatend,noinfoline]{article}
\usepackage{fullpage}
\usepackage[english]{babel}
\usepackage{amssymb}
\usepackage{amsmath}
\usepackage{theorem}
\usepackage{epsfig}
\usepackage{color}
\usepackage{imsart}
\usepackage{comment}
\usepackage{amsmath}

\newtheorem{theorem}{Theorem}
\newtheorem{lemma}[theorem]{Lemma}

\newenvironment{proof}{\noindent{\bf Proof.}}{\hspace*{2mm}~$\square$}

\newcommand{\Z}{\mathbb{Z}}
\newcommand{\R}{\mathbb{R}}
\newcommand{\norm}[1]{|\!|#1|\!|}
\newcommand{\ind}{\mathbf{1}}
\newcommand{\ep}{\epsilon}
\newcommand{\n}{\hspace*{-5pt}}

\DeclareMathOperator{\card}{card}

\begin{document}

\begin{frontmatter}
\title     {Deffuant opinion dynamics with \\ attraction and repulsion}
\runtitle  {Deffuant opinion dynamics with attraction and repulsion}
\author    {Nicolas Lanchier\thanks{Nicolas Lanchier was partially supported by NSF grant CNS-2000792.} and Max Mercer}
\runauthor {Nicolas Lanchier and Max Mercer}
\address   {School of Mathematical and Statistical Sciences \\ Arizona State University \\ Tempe, AZ 85287, USA. \\ nicolas.lanchier@asu.edu \\ mamerce1@asu.edu}

\maketitle

\begin{abstract} \ \
 In the Deffuant model, individuals are located on the vertices of a graph, and are characterized by their opinion, a number in $[-1, 1]$. The dynamics depends on two parameters: a confidence threshold $\theta < 2$ and a convergent parameter $\mu_- \leq 1/2$. Neighbors on the graph interact at rate one, which results in no changes if the neighbors disagree by more than $\theta$, and a compromise with the opinions moving toward each other by a factor $\mu_-$ if they disagree by less than $\theta$ (attraction). The main conjecture about the Deffuant model, which was proved for the process on the integers, states that, for all $\mu_- > 0$ and starting from the product measure in which the opinions are uniformly distributed in the interval $[-1, 1]$, there is a phase transition from discordance to consensus at the confidence threshold one. In this paper, we study a natural variant of the model in which neighbors who disagree by more than $\theta$ feel more strongly about their own opinion, which is modeled by assuming that the opinions move away from each other by a divergent parameter $\mu_+$ (repulsion). We prove, for the process on the integers, the absence of a phase transition even for arbitrarily small $\mu_+ > 0$, in the sense that, for every nontrivial choice of $\theta$, there is always discordance.
\end{abstract}

\begin{keyword}[class=AMS]
\kwd[Primary ]{60K35 }
\kwd[Secondary ]{91D25, 91D30}
\end{keyword}

\begin{keyword}
\kwd{Interacting particle systems, opinion dynamics, Deffuant model, attraction, repulsion.}
\end{keyword}

\end{frontmatter}


\section{Introduction}
 Homophily and social influence have been identified by sociologists as the two main mechanisms that govern the dynamics of opinions.
 Homophily is defined as the tendency of individuals to interact more when they are similar, while social influence is defined as the tendency of individuals to become more similar when they interact.
 In short, more similarities leads to more interactions, which leads to more similarities, and so on.
 One of the main questions in the context of opinion dynamics is whether these two mechanisms combined can lead an entire population to a global consensus or disagreements persist in the long run. \\
\indent
 The first opinion model based on interacting particle systems is the voter model, which was introduced independently by Holley and Liggett~\cite{holley_liggett}, and Clifford and Sudbury~\cite{clifford_sudbury}.
 In the voter model, individuals characterized by two possible competing opinions are located on the vertices of a graph interpreted as a social network, and update their opinion at rate one by mimicking one of their neighbors chosen uniformly at random.
 In particular, the process includes social influence but not homophily.
 The main result about the voter model shows that, when the underlying graph is the~$d$-dimensional integer lattice, clusters of the same opinion get larger and larger in dimensions one and two~(consensus), whereas the system converges to an invariant measure in which both opinions coexist in higher dimensions~(discordance).
 This dichotomy is due to the duality relationship between the voter model and coalescing random walks. \\
\indent
 The first model of interacting particle systems that includes both homophily and social influence is the popular Axelrod model~\cite{axelrod}.
 In this process, individuals are characterized by a vector of opinions about different issues, and homophily is modeled by assuming that neighbors on the graph interact at a rate proportional to the number of opinions they have in common.
 As a result of an interaction, one of the two neighbors chosen at random updates its opinion about one of the issues to match the opinion of the other neighbor on that issue, which models social influence.
 The Deffuant model~\cite{deffuant} is another popular model of opinion dynamics that includes both homophily and social influence.
 Homophily takes the form of a confidence threshold, and social influence takes the form of a mutual compromise.
 More precisely, the set of opinions now consists of a compact interval, say~$[-1, 1]$, so the state at time~$t$ is a configuration
 $$ \xi_t : V \longrightarrow [-1, 1], \quad \hbox{where} \quad  \xi_t (x) = \hbox{opinion at site~$x$ at time~$t$}, $$
 and where~$V$ refers to the vertex set of the underlying graph.
 The model depends on two parameters: a confidence threshold~$\theta$ measuring the strength of homophily, and a convergent parameter~$\mu$ measuring the strength of social influence.
 Declaring two individuals to be compatible if the distance between their opinions does not exceed the confidence threshold~$\theta$, neighbors on the graph interact at rate one, which results in their opinions moving toward each other by the factor~$\mu$~(social influence) if the two neighbors are compatible~(homophily), whereas nothing happens if the two neighbors are not compatible.
 More precisely, an interaction at time~$t$ between two neighbors~$x$ and~$y$ with respective opinions~$a$ and~$b$ at time~$t-$ results in the pair of simultaneous jumps
  $$ \begin{array}{rcl}
    \xi_{t-} (x) = a & \n \to \n & \xi_t (x) = (a + \mu (b - a)) \ \ind \{|a - b| \leq \theta \},  \vspace*{4pt} \\
    \xi_{t-} (y) = b & \n \to \n & \xi_t (y) = (b + \mu (a - b)) \ \ind \{|a - b| \leq \theta \}. \end{array} $$
 Based on numerical simulations of the process on various graphs, Deffuant et al.~\cite{deffuant} conjectured the presence of a phase transition at a critical threshold:
 starting from the configuration in which the opinions are independent, uniformly distributed in the interval~$[-1, 1]$, the system converges to a consensus if~$\theta > 1$, whereas disagreements persist if~$\theta < 1$, in the sense that
 $$ \begin{array}{rcll}
      P (\lim_{t \to \infty} |\xi_t (x) - \xi_t (y)| = 0) = 1 & \hbox{when} & \theta > 1 & (\hbox{consensus}), \vspace*{4pt} \\
      P (\liminf_{t \to \infty} |\xi_t (x) - \xi_t (y)| > \theta) > 0 & \hbox{when} & \theta < 1 & (\hbox{discordance}), \end{array} $$
 for all neighbors~$x$ and~$y$.
 Though the parameter~$\mu$ affects the rate of convergence of the process, simulations also suggest that it does not affect the limiting behavior. \\
\indent
 The existence of a phase transition from discordance to consensus at the critical threshold one was first proved analytically by Lanchier~\cite{lanchier2012} for the process on the integers in the presence of nearest neighbor interactions.
 Shortly after, a simpler and more transparent alternative proof of this result was proposed by H\"aggstr\"om~\cite{Haggstrom180}.
 Following the works of Lanchier and H\"aggstr\"om, a number of variants of the Deffuant model with a more general initial distribution and/or opinion space and/or social network were also studied analytically. 
 H\"aggstr\"om and Hirscher~\cite{HagHir181} studied the process on the integers starting from the configuration in which the opinions are independent and identically distributed, but not necessarily uniform across the opinion space.
 Let~$[a, b]$ be the smallest compact interval that contains the support of the initial opinion distribution, and let~$L$ be the length of the gap~$E (\xi_0 (x))$ belongs to~(with the convention~$L = 0$ if the expected value belongs to the support).
 H\"aggstr\"om and Hirscher proved that the phase transition from discordance to consensus now occurs at the critical threshold
 $$ \theta_c = \max (E(\xi_0(x)) - a, b - E(\xi_0(x)), L). $$
 Along these lines, they also proved that, if the support is not bounded, the phase transition disappears: there is discordance for all confidence thresholds.
 Hirscher~\cite{hirscher2014deffuant} generalized these results to the multivariate case where the opinions are vectors in the space~$\R^n$, and where the disagreements are measured using the Euclidean distance.
 To state his results, let
 $$ r_1 = \inf \{r > 0 : P (\xi_0 (x) \in B (E (\xi_0 (x)), r)) = 1 \} $$
 be the radius of the smallest ball centered at~$E (\xi_0 (x))$ that contains the support.
 Like in the univariate case, at least when the distribution of the initial opinions has mass around its mean, there is a phase transition from discordance to consensus at the critical threshold~$\theta_c = r_1$ when the support is bounded, whereas there is always discordance when the support is unbounded.
 Hirscher~\cite{hirscher2016overly} also studied the process on the integers where the opinions are probability distributions which are absolutely continuous with respect to the Lebesgue measure, thus characterized by some density function with support in~$[-1, 1]$.
 In this case, the disagreement between individuals is measured using the total variation distance, which is bounded by one.
 Specializing in the symmetric triangular distribution, with the support at time zero being chosen by selecting two independent uniform random variables in~$[-1, 1]$ for each integer~$x \in \Z$, Hirscher proved the absence of a phase transition, with always discordance when~$\theta < 1$.
 He also proved, however, the presence of a phase transition when imposing that the length of the support be larger than a fixed positive constant.
 Gantert, Haydenreich and Hirscher~\cite{gantert2019strictly} studied another variant of the Deffuant model on the integers in which the opinion space is path-connected but not simply connected.
 In this case, the disagreements between opinions are measured using the length of the geodesic connecting them, while interactions move the opinions along the geodesic.
 Specializing in the case where the opinion space consists of the unit circle, and ignoring homophily by setting~$\theta = \infty$, they proved that weak consensus~(neighbors share asymptotically the same opinion) and strong consensus~(all the individuals share asymptotically the same opinion) are not equivalent.
 More precisely, letting~$d_G$ denote the geodesic distance on the unit circle, and assuming that the initial opinions are independent and uniformly distributed, there is weak consensus~(in mean), in the sense that
 $$ \begin{array}{l} \lim_{t \to \infty} E (d_G (\xi_t (x), \xi_t (x + 1))) = 0 \quad \hbox{for all} \quad x \in \Z, \end{array} $$
 but not strong consensus~(in probability), in the sense that there exists~$\ep > 0$ such that, for all possible opinions~$c$ on the unit circle,
 $$ \begin{array}{l} \limsup_{t \to \infty} P (d_G (\xi_t(x), c) > \ep) > 0 \quad \hbox{for some} \quad x \in \Z. \end{array} $$
 The previous results rely, among other techniques, on a notion introduced by Lanchier~\cite{lanchier2012} and named~$\ep$-flatness by H\"aggstr\"om~\cite{Haggstrom180}, that can be defined as follows:
 an integer~$x$ is said to be~$\ep$-flat if the opinion distance between its initial opinion and the initial opinion of any other integer~$y$ does not exceed~$\ep |x - y|$.
 Keeping track of the integers that are~$\ep$-flat is the key to studying the process on the integers but this approach fails when dealing with other networks of interactions.
 However, there have been a few attempts at studying the process on more general graphs/social networks.
 H\"aggstr\"om and Hirscher~\cite{HagHir181} looked at the original Deffuant model on the~$d$-dimensional lattice, and used the concept of energy to prove consensus for all~$\theta > 3/2$.
 Lanchier and Li~\cite{lanchier_li} also studied a multivariate version of the model in which the opinion space is a bounded convex subset~$\Delta$ of a normed vector space, and the network of interactions a general finite connected graph.
 Due to the finiteness of the graph, there is no phase transition like on the integers, and the main objective is to estimate the probability of consensus.
 Let~$\norm{\cdot}$ be the norm on the opinion space, and
 $$ r_2 = \inf \{r > 0 : \Delta \subset B (c, r) \ \hbox{for some} \ c \in \Delta \}, $$
 which we interpret as the radius of the opinion space, and fix~$c_0$ such that~$\Delta \subset B (c_0, r_2)$.
 Using techniques from martingale theory, Lanchier and Li proved that
 $$ P (\hbox{consensus}) \geq 1 - E \norm{\xi_0 (x) - c_0)} / (\theta - r_2) \quad \hbox{for all} \quad \theta > r_2, $$
 when starting from the configuration in which the opinions are independent, and uniformly distributed in the opinion space~$\Delta$.
 For a detailed review of the Deffuant model and closely related opinion models, we refer the reader to~\cite[Section~6.4]{IPS}.


\section{Model description and main results}
 The main objective of this paper is to introduce and study another natural variant of the original Deffuant model~\cite{deffuant}.
 Following~\cite{Haggstrom180,lanchier2012}, we assume that the individuals are located on the integers, and interact with each of their two nearest neighbors at rate one.
 The initial opinions are again independent and uniformly distributed in the interval~$[-1, 1]$.
 Like in the original Deffuant model, interactions between compatible neighbors~(whose opinion distance does not exceed the confidence threshold~$\theta < 2$) result in the neighbors' opinions moving toward each other by some fixed convergent parameter~$\mu_- \leq 1/2$, which we refer to as attraction.
 In addition, we now assume that interactions between incompatible neighbors~(whose opinion distance exceeds the confidence threshold~$\theta$) also affect the configuration of the system but now result in the neighbors' opinions moving away from each other by some fixed divergent parameter~$\mu_+ \leq 1/2$, which we refer to as repulsion.
 In short, through an interaction, individuals who somewhat agree become even more similar, whereas individuals who somewhat disagree become even more dissimilar:
 an interaction results in an escalated fight that makes the two opponents even more extremist in their position.
 Due to the presence of repulsion, even when starting with all the opinions in~$[-1, 1]$, the opinions can move away from this interval, so the state at time~$t$ is a configuration
 $$ \xi_t : \Z \longrightarrow \R, \quad \hbox{where} \quad  \xi_t (x) = \hbox{opinion at site~$x$ at time~$t$}. $$
 In addition, an interaction at time~$t$ between two neighbors~$x$ and~$y$ with respective opinions~$a$ and~$b$ at time~$t-$ now results in the pair of simultaneous jumps
 $$ \begin{array}{rcl}
    \xi_{t-} (x) = a & \n \to \n & \xi_t (x) = (a + \mu_- (b - a)) \ \ind \{|a - b| \leq \theta \} + (a - \mu_+ (b - a)) \ \ind \{|a - b| > \theta \}, \vspace*{4pt} \\
    \xi_{t-} (y) = b & \n \to \n & \xi_t (y) = (b + \mu_- (a - b)) \ \ind \{|a - b| \leq \theta \} + (b - \mu_+ (a - b)) \ \ind \{|a - b| > \theta \}. \end{array} $$
 Note that the original Deffuant model~\cite{deffuant} reduces to the particular case where the convergent parameter~$\mu_- = \mu > 0$, and where the divergent parameter~$\mu_+ = 0$.
 As previously, the main objective is to determine whether or not the process starting from the product measure in which the opinions are uniformly distributed in~$[-1, 1]$ exhibits a phase transition from discordance to weak/strong consensus at some nondegenerate critical threshold~$\theta_c \in (0, 2)$. \\
\indent
 To get some insight on the possible behavior of the spatial stochastic process, we first look at its mean-field approximation, the nonspatial deterministic model obtained by assuming that sites are independent and the system spatially homogeneous.
 In the case of the Deffuant model, because the opinion space is uncountable, the mean-field model consists of an integro-differential equation for the density~$u (a, t)$ of individuals with opinion~$a$ at time~$t$ rather than a system of ordinary differential equations.
 More precisely, assuming for simplicity that~$\mu_- = 1/2$,
\begin{equation}
\label{eq:mean-field}
\begin{array}{rcl}
\displaystyle \frac{\partial u (a, t)}{\partial t} & \n = \n &
\displaystyle \int_{[- \theta / 2, \theta / 2]} u (a + b, t) \,u (a - b, t) \,db \vspace*{8pt} \\ & \n + \n &
\displaystyle \int_{\R \setminus [- \theta \mu_+, \theta \mu_+]} u (a + b, t) \,u (a + b + b / \mu_+, t) \,db - u (a, t). \end{array}
\end{equation}
 The first integral comes from the fact that an interaction~(attraction) between two compatible individuals with opinions~$a - b$ and~$a + b$, with~$|b| \leq \theta / 2$, results in both individuals having new opinion~$a$.
 Similarly, the second integral comes from the fact that an interaction~(repulsion) between two incompatible individuals with opinions~$a + b$ and~$a + b + b / \mu_+$, with~$|b| > \theta \mu_+$, results in the first individual changing its opinion to
 $$ (a + b) - \mu_+ ((a + b + b / \mu_+) - (a + b)) = a + b - \mu_+ b / \mu_+ = a. $$
 Finally, the last term~$- u (a, t)$ simply comes from the fact that individuals with opinion~$a$ change their opinion when they interact.
 Assume that there exist~$c_0 \geq \theta / 2$ and~$\ep > 0$ such that
\begin{equation}
\label{eq:base-case}
  u (c, t) \geq \ep \quad \hbox{for all} \quad c \in I_{\ep} = [- c_0 - 2 \ep, - c_0] \cup [c_0, c_0 + 2 \ep],
\end{equation}
 and let~$c_+ = (2 \mu_+ + 1)(c_0 + \ep)$.
 Then, for all
 $$ b \in J_{\ep} = [- 2 \mu_+ (c_0 + \ep) - \ep / (1 + 1 / \mu_+), - 2 \mu_+ (c_0 + \ep) + \ep / (1 + 1 / \mu_+)], $$
 some basic algebra shows that~$c_+ + b \in I_{\ep}$ and~$c_+ + b + b / \mu_+ \in I_{\ep}$.
 In particular, it follows from the mean-field dynamics~\eqref{eq:mean-field} and assumption~\eqref{eq:base-case} that, whenever~$u (c_+, t) = 0$,
 $$ \begin{array}{rcl}
    \displaystyle \frac{\partial u (c_+, t)}{\partial t} & \n \geq \n &
    \displaystyle \int_{\R \setminus [- \theta \mu_+, \theta \mu_+]} u (c_+ + b, t) \,u (c_+ + b + b / \mu_+, t) \,db \vspace*{8pt} \\ & \n \geq \n &
    \displaystyle \int_{J_{\ep}} u (c_+ + b, t) \,u (c_+ + b + b / \mu_+, t) \,db \geq \int_{J_{\ep}} \ep^2 \,db \geq \frac{2 \ep^3}{1 + 1 / \mu_+} > 0, \end{array} $$
 showing that the dynamics produces a positive density around opinion~$c_+$.
 By symmetry, the same holds around opinion~$c_- = - (2 \mu_+ + 1)(c_0 + \ep)$.
 Now, starting from the configuration in which the opinions are independent and uniformly distributed in the interval~$[-1, 1]$, we have the initial profile~$u (a, 0) = 1/2$ for all~$a \in [-1, 1]$, therefore, for all~$\theta < 2$,
 $$ u (c, 0) \geq 1/2 - \theta / 4 \quad \hbox{for all} \quad c \in [- 1, - \theta / 2] \cup [\theta / 2, 1], $$
 which shows that~\eqref{eq:base-case} holds at time zero for~$c_0 = \theta / 2$ and~$\ep = 1/2 - \theta / 4$.
 Using a simple induction, we deduce that, for all~$c > 0$, there is a positive density of individuals with opinion larger than~$c$ and a positive density of individuals with opinion smaller than~$-c$ at large enough times. \\
\indent
 The main objective of this paper is to prove that the interacting particle systems exhibits the same behavior:
 absence of a phase transition, with discordance for all confidence thresholds~$\theta < 2$ even for small perturbations of the Deffuant model with~$\mu_+ > 0$ close to zero, and the existence of opinions drifting arbitrarily far from each other.
 More precisely, we prove that
 \begin{theorem}
\label{th:divergence}
 For all~$\theta < 2$ and all~$\mu_+ > 0$,
 $$ \begin{array}{l} \card \,\{x \in \Z : \lim_{t \to \infty} |\xi_t (x + 1) - \xi_t (x)| = \infty \} \geq 1. \end{array} $$
\end{theorem}
 Our strategy to prove the theorem is to construct an infinite collection of processes that keep track of edges along which the opinion distance diverges to infinity.
 Because these processes might intersect, in which case they coalesce, our proof only guarantees the existence of at least one edge along which the opinion distance diverges to infinity, as time goes to infinity.
 We conjecture, however, that not only the density of such edges is positive, thus making the cardinal in the theorem infinite, but also that the opinion distance along \emph{all} the edges diverges to infinity.


\section{Proof of Theorem~\ref{th:divergence}}
 To study the opinion model, it is convenient to keep track of the disagreement along the edges~(which we also call the gap along the edges for short) rather than the opinion at the vertices.
 More precisely, identifying each edge~$e = (x, x + 1)$ with its midpoint~$x + 1/2$, we let
 $$ \bar \xi_t (e) = |\xi_t (e + 1/2) - \xi_t (e - 1/2)| \quad \hbox{for all} \quad e \in \Z + 1/2 $$
 be the gap along~$e$ at time~$t$.
 To state our results, we also define
 $$ \rho_- = \frac{1 + 2 \mu_+}{1 + 3 \mu_+} < 1, \qquad \rho_+ = 1 + 2 \mu_+ > 1, \qquad D = \frac{\mu_- \theta}{1 - \rho_-} = \bigg(3 + \frac{1}{\mu_+} \bigg) \mu_- \theta. $$
 The proof of Theorem~\ref{th:divergence} relies on the analysis of a system of processes that keep track of a large gap.
 For each edge~$e$, we consider a collection of processes~$\zeta_t (e)$ starting at~$\zeta_0 (e) = e$.
 To describe their dynamics, assume that~$\zeta_s (e)$ has been defined until time~$t-$ and that
\begin{equation}
\label{eq:interaction}
\zeta_{t-} (e) = e' \quad \hbox{and} \quad \hbox{there is an interaction along~$e' - 1$, $e'$ or~$e' + 1$ at time~$t$}.
\end{equation}
 Then, the process jumps at time~$t$ to the edge with the largest gap:
 $$ \zeta_t (e) = \hbox{unique~$e'' \in \Z + 1/2$ such that} \ \bar \xi_t (e'') = \max (\bar \xi_t (e' - 1), \bar \xi_t (e'), \bar \xi_t (e' + 1)). $$
 To prove the theorem, the first step is to show that we can, with positive probability, keep the process~$\zeta_t (e)$ at edge~$e$ for at least one unit of time while increasing the gap along edge~$e$ up to the point that it exceeds~$2D$ at time one~(see Lemma~\ref{lem:increase-gap}).
 In particular, $\zeta_1 (e) = e$ and
\begin{equation}
\label{eq:timeone}
\bar \xi_1 (\zeta_1 (e)) > 2D \quad \hbox{with positive probability}.
\end{equation}
 The next step consists in showing that, whenever~\eqref{eq:interaction} holds and the gap along~$e'$ just before time~$t$ exceeds~$D$, one of the following two cases occurs:
\begin{itemize}
\item The interaction at time~$t$ is along edge~$e' - 1$ or along edge~$e' + 1$, which occurs at rate two, in which case the largest gap along the trio of edges~$e' - 1, e'$ and~$e' + 1$ after the interaction is at least~$\rho_- \bar \xi_{t-} (e')$~(see Lemma~\ref{lem:control-gap}). \vspace*{4pt}
\item The interaction at time~$t$ is along edge~$e'$, which occurs at rate one, in which case the new gap along edge~$e'$ increases by the factor~$\rho_+$.
\end{itemize}
 This shows that, as long as the gap along~$\zeta_t (e)$ exceeds~$D$, it dominates stochastically the process~$X_t$ with state space~$\R_+$ and transition rates
\begin{equation}
\label{eq:submartingale}
  X_t \to \rho_- X_t \quad \hbox{at rate~2} \qquad \hbox{and} \qquad X_t \to \rho_+ X_t \quad \hbox{at rate~1}.
\end{equation}
 The analysis of this process shows that, when starting above~$2D$, the process converges to infinity while staying above~$D$ at all times, with positive probability~(see Lemma~\ref{lem:OST}).
 This, together with~\eqref{eq:timeone} and the stochastic domination of~\eqref{eq:submartingale}, implies that
\begin{equation}
\label{eq:divergence}
 P (\bar \xi_t (\zeta_t (e)) \to \infty) = p (\theta, \rho_-, \rho_+) > 0 \quad \hbox{for all} \quad e \in \Z + 1/2.
\end{equation}
 Since the initial distribution and the evolution rules~(the graphical representation of the opinion model) are translation invariant, it follows from~\eqref{eq:divergence} and the ergodic theorem that
 $$ \begin{array}{l} \card \{e \in \Z + 1/2 : \lim_{t \to \infty} \bar \xi_t (\zeta_t (e)) = \infty \} = \infty \quad \hbox{with probability one}. \end{array} $$
 Because the processes~$\zeta_t (e)$, $e \in \Z + 1/2$, might intersect and then coalesce, this shows that there is at least one edge~(not infinitely many) along which the gap diverges to infinity.
 The rest of this section is devoted to proving the technical details leading to~\eqref{eq:timeone}--\eqref{eq:divergence}. \\
\indent
 To prove~\eqref{eq:timeone}, the idea is to start with a large gap along~$e$ and small gaps along~$e \pm 1$, then impose enough interactions along~$e$ to increase the gap along~$e$, but no interactions along the four nearest edges to also make sure that the process~$\zeta_t (e)$ stays at edge~$e$ until time one.
\begin{lemma}
\label{lem:initial-gap}
 Assume that~$\theta < 2$.
 Then, for all~$e \in \Z + 1/2$,
 $$ P (\bar \xi_0 (e \pm 1) < \theta < \bar \xi_0 (e)) \geq (1/2 - \theta / 4)^2 (\theta/ 4)^2 > 0. $$
\end{lemma}
\begin{proof}
 By inclusion of events, and recalling that the random variables~$\xi_0 (x)$, $x \in \Z$, are independent and uniformly distributed, we get the lower bound
 $$ \begin{array}{l}
     P (\bar \xi_0 (e \pm 1) < \theta < \bar \xi_0 (e)) \vspace*{4pt} \\ \hspace*{10pt} \geq
     P (\xi_0 (e - 1/2) > + \theta / 2, \,\xi_0 (e + 1/2) < - \theta / 2, \vspace*{4pt} \\ \hspace*{50pt}
     0 > \xi_0 (e - 3/2) - \xi_0 (e - 1/2) > - \theta / 2, \,0 < \xi_0 (e + 3/2) - \xi_0 (e + 1/2) < + \theta / 2) \vspace*{4pt} \\ \hspace*{10pt} =
     P (\xi_0 (e - 1/2) > + \theta / 2) \,P (\xi_0 (e + 1/2) < - \theta / 2) \vspace*{4pt} \\ \hspace*{22pt}
     P (0 > \xi_0 (e - 3/2) - \xi_0 (e - 1/2) > - \theta / 2 \,| \,\xi_0 (e - 1/2) > + \theta / 2)  \vspace*{4pt} \\ \hspace*{22pt}
     P (0 < \xi_0 (e + 3/2) - \xi_0 (e + 1/2) < + \theta / 2 \,| \,\xi_0 (e + 1/2) < - \theta / 2) =
     (1/2 - \theta / 4)^2 (\theta / 4)^2. \end{array} $$
 This completes the proof.
\end{proof}
\begin{lemma}
\label{lem:increase-gap}
 Assume that~$\theta < 2$.
 Then, for all~$e \in \Z + 1/2$,
 $$ P (\bar \xi_1 (\zeta_1 (e)) > 2D) \geq (1/2 - \theta / 4)^2 (\theta/ 4)^2 (e^{-5} / K!) > 0, $$
 where~$K = \lceil \log (2D / \theta) / \log (\rho_+) \rceil$.
\end{lemma}
\begin{proof}
 To begin with, observe that we have the following:
\begin{itemize}
\item As long as there are no interactions along the four edges~$e \pm 1$, $e \pm 2$, the opinions at the two vertices~$e - 3/2$ and~$e + 3/2$ remain the same. \vspace*{4pt}
\item As long as the gap along edge~$e$ exceeds the threshold~$\theta$, each interaction along that edge applies a~$\rho_+$ multiplier to the opinion distance~(repulsion). \vspace*{4pt}
\item As long as the gap along edge~$e$ exceeds both the threshold~$\theta$ and the gaps along edges~$e \pm 1$, interactions along edge~$e$ do not change this property.
\end{itemize}
 Letting~$I_e$ be the number of interactions along edge~$e$ before time one, and combining the three properties above, we obtain the inclusions of events
\begin{equation}
\label{eq:increase-gap-1}
\begin{array}{l}
\{\bar \xi_0 (e \pm 1) < \theta < \bar \xi_0 (e) \} \cap \{I_{e \pm 1} = I_{e \pm 2} = 0 \} \cap \{I_e = K \} \vspace*{4pt} \\ \hspace*{20pt}
\subset \{\bar \xi_0 (e \pm 1) < \theta < \bar \xi_0 (e) \} \cap \{\xi_t (e \pm 3/2) = \xi_0 (e \pm 3/2) \ \hbox{for all} \ t \leq 1 \} \cap \{I_e = K \} \vspace*{4pt} \\ \hspace*{20pt}
\subset \{\zeta_1 (e) = e \} \cap \{\bar \xi_1 (e) > \theta \rho_+^K \} \subset \{\bar \xi_1 (\zeta_1 (e)) > \theta \rho_+^K \}. \end{array}
\end{equation}
 Notice also that, in view of the expression of~$K$,
\begin{equation}
\label{eq:increase-gap-2}
\log (\rho_+^K) = K \log (\rho_+) \geq \log (2D / \theta) \quad \hbox{therefore} \quad \theta \rho_+^K \geq 2D.
\end{equation}
 In addition, because the five random variables~$I_e, I_{e \pm 1}, I_{e \pm 2}$ are independent Poisson random variables with parameter one, we also have
\begin{equation}
\label{eq:increase-gap-3}
  P (I_{e \pm 1} = I_{e \pm 2} = 0, \,I_e = K) = e^{-1} \cdot e^{-1} \cdot e^{-1} \cdot e^{-1} \cdot e^{-1} / K! = e^{-5} / K! > 0.
\end{equation}
 Combining~\eqref{eq:increase-gap-1}--\eqref{eq:increase-gap-3} and Lemma~\ref{lem:initial-gap}, we conclude that
 $$ \begin{array}{rcl}
      P (\bar \xi_1 (\zeta_1 (e)) > 2D) & \n \geq \n &
      P (\bar \xi_1 (\zeta_1 (e)) > \theta \rho_+^K) \vspace*{4pt} \\ & \n \geq \n &
      P (\bar \xi_0 (e \pm 1) < \theta < \bar \xi_0 (e)) \,P (I_{e \pm 1} = I_{e \pm 2} = 0, \,I_e = K) \vspace*{4pt} \\ & \n \geq \n &
      (1/2 - \theta / 4)^2 (\theta / 4)^2 (e^{-5} / K!) > 0, \end{array} $$
 which completes the proof of the lemma, and shows~\eqref{eq:timeone}.
\end{proof} \\ \\
 We now look at the dynamics of the gap when the gap is sufficiently large.
 The evolution rules of the opinion model imply that, when the gap at~$\zeta_t (e)$ exceeds the confidence threshold~$\theta$, interactions along~$\zeta_t (e)$ occur at rate one, and increase the gap by a factor~$\rho_+$.
 The gap at~$\zeta_t (e)$ can also change due to interactions along the edges~$\zeta_t (e) \pm 1$, which occur at rate two.
 In particular, to prove that, as long as the gap at~$\zeta_t (e)$ exceeds~$D$, it dominates the process~$X_t$ defined in~\eqref{eq:submartingale}, it suffices to prove that, for every edge~$e \in \Z + 1/2$, if there is an interaction along~$e \pm 1$ at time~$t$, then
\begin{equation}
\label{eq:control-gap}
\bar \xi_{t-} (e) > D \quad \hbox{implies that} \quad \max (\bar \xi_t (e - 1), \bar \xi_t (e), \bar \xi_t (e + 1)) \geq \rho_- \,\bar \xi_{t -} (e).
\end{equation}
 To prove~\eqref{eq:control-gap}, we start with the following two technical lemmas.
\begin{lemma}
\label{lem:D}
 For all~$g > D$, we have~$g - \mu_- \theta > \rho_- g$.
\end{lemma}
\begin{proof}
 Using that~$\rho_- < 1$, we get
 $$ g - \mu_- \theta - \rho_- g = (1 - \rho_-) g - \mu_- \theta > (1 - \rho_-) D - \mu_- \theta = 0. $$
 This completes the proof.
\end{proof}
\begin{lemma}
\label{lem:align}
 Let~$a > b > c$ and set~$b' = b - \mu_+ (c - b)$ and~$c' = c - \mu_+ (b - c)$. Then,
\begin{equation}
\label{eq:align}
\frac{|a - b'| \vee |b' - c'|}{|a - b|} \geq \rho_-.
\end{equation}
\end{lemma}
\begin{proof}
 Intuitively/geometrically,~$a$ and~$b$ being fixed, the left-hand side of~\eqref{eq:align} reaches its minimum when the three points~$(-1, a)$, $(0, b')$, and~$(1, c')$ are aligned.
 To prove the result analytically, we distinguish between the following two cases. \vspace*{5pt} \\
{\bf Case 1}. $(a - b) \leq \mu_+ (b - c)$.
 In this case,
 $$ b' - c' = (b - c) + 2 \mu_+ (b - c) \geq (b - c) + 2 (a - b) \geq a - b \geq 0, $$
 from which it follows that
 $$ \frac{|a - b'| \vee |b' - c'|}{|a - b|} \geq \frac{b' - c'}{a - b} \geq 1 \geq \rho_- \quad \hbox{for all} \quad \mu_+ > 0. $$
{\bf Case 2}. $(a - b) \geq \mu_+ (b - c)$.
 In this case,
 $$ \begin{array}{rcrcll}
    \phi_1 (c) & \n = \n & \displaystyle \frac{a - b'}{a - b}  & \n = \n & \displaystyle \frac{(a - b) - \mu_+ (b - c)}{a - b} \geq 0, \vspace*{8pt} \\
    \phi_2 (c) & \n = \n & \displaystyle \frac{b' - c'}{a - b} & \n = \n & \displaystyle \frac{(1 + 2 \mu_+)(b - c)}{a - b} \geq 0, \end{array} $$
 and, for every fixed~$a$,~$b$, and~$\mu_+$, the graphs of~$\phi_1$ and~$\phi_2$ are straight lines of strictly positive and strictly negative slope, respectively.
 In particular, the minimum of~$\phi_1 (c) \vee \phi_2 (c)$ is reached when~$\phi_1 (c) = \phi_2 (c)$, which gives the condition
 $$ \begin{array}{rcl}
      0 & \n = \n & (a - b') - (b' - c') = a - 2b' + c' = a - 2b - 2 \mu_+ (b - c) + c - \mu_+ (b - c) \vspace*{4pt} \\
        & \n = \n & (a - b) - (1 + 3 \mu_+)(b - c). \end{array} $$ 
 In addition, for~$c_0$ such that~$a - b = (1 + 3 \mu_+)(b - c_0)$,
 $$ \begin{array}{rcl}
    \phi_1 (c_0) \vee \phi_2 (c_0) & \n = \n & \phi_2 (c_0) = \displaystyle \frac{(1 + 2 \mu_+)(b - c_0)}{a - b} = \frac{(1 + 2 \mu_+)(b - c_0)}{(1 + 3 \mu_+)(b - c_0)} = \rho_-. \end{array} $$
 This completes the proof.
\end{proof} \\ \\
 Using the previous two lemmas, we can now prove the implication~\eqref{eq:control-gap}.
\begin{lemma}
\label{lem:control-gap}
 Assume that there is an interaction along~$e \pm 1$ at time~$t$.
 Then~\eqref{eq:control-gap} holds.
\end{lemma}
\begin{figure}[t!]
\centering
\scalebox{0.30}{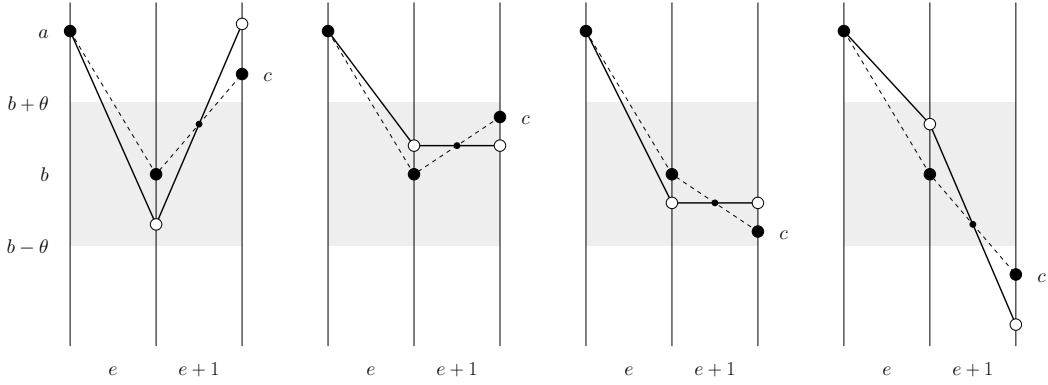}
\caption{\upshape{
 Result of an interaction along~$e + 1$ connecting vertices with opinions~$b$ and~$c$ depending on whether~$c$ is smaller/larger than~$b$ and within the confidence threshold of~$b$ or not.
 The black dots represent the opinions before the interaction, the white dots the opinions after the interaction, and the gray area the confidence area of~$b$.}}
\label{fig:interaction}
\end{figure}
\begin{proof}
 By symmetry, we may assume that the interaction occurs along edge $e + 1$.
 To simplify the notations, and underline the connection with the previous lemma, we let
 $$ \xi_{t-} (e - 1/2) = a, \qquad \xi_{t-} (e + 1/2) = b, \qquad \xi_{t-} (e + 3/2) = c. $$
 By symmetry, we may also assume that~$a > b$.
 To prove the lemma, we distinguish among four cases depending on whether~$c$ is larger or smaller than~$b$, and depending on whether the opinion distance~$|b - c|$ is larger or smaller than the confidence threshold~$\theta$, as shown in Figure~\ref{fig:interaction}. \vspace*{5pt} \\
{\bf Case 1}. $c \in (b + \theta, \infty)$. In this case, there is repulsion so
 $$ \begin{array}{rcl}
    \bar \xi_t (e) & \n = \n & |\xi_{t-} (e - 1/2) - \xi_t (e + 1/2)| = |a - (b - \mu_+ (c - b))| \vspace*{4pt} \\
                   & \n = \n & |\bar \xi_{t-} (e) + \mu_+ (c - b)| = \bar \xi_{t-} (e) + \mu_+ (c - b) \geq \bar \xi_{t-} (e) \geq \rho_- \,\bar \xi_{t-} (e). \end{array} $$
{\bf Case 2}. $c \in [b, b + \theta]$. In this case, there is attraction. Using also that
 $$ \bar \xi_{t-} (e) > D > \mu_- \theta \geq \mu_- (c - b), $$
 and applying Lemma~\ref{lem:D}, we deduce that
 $$ \begin{array}{rcl}
    \bar \xi_t (e) & \n = \n & |\xi_{t-} (e - 1/2) - \xi_t (e + 1/2)| = |a - (b + \mu_- (c - b))| \vspace*{4pt} \\
                   & \n = \n & |\bar \xi_{t-} (e) - \mu_- (c - b)| = \bar \xi_{t-} (e) - \mu_- (c - b) \geq \bar \xi_{t-} (e) - \mu_- \theta \geq \rho_- \,\bar \xi_{t-} (e). \end{array} $$
{\bf Case 3}. $c \in [b - \theta, b]$. In this case, there is attraction so
$$ \begin{array}{rcl}
   \bar \xi_t (e) & \n = \n & |\xi_{t-} (e - 1/2) - \xi_t (e + 1/2)| = |a - (b + \mu_- (c - b))| \vspace*{4pt} \\
                  & \n = \n & |\bar \xi_{t-} (e) - \mu_- (c - b)| = \bar \xi_{t-} (e) + \mu_- (b - c) \geq \bar \xi_{t-} (e) \geq \rho_- \,\bar \xi_{t-} (e). \end{array} $$
{\bf Case 4}. $c \in (- \infty, b - \theta)$. In this case, there is repulsion. Using also that~$a > b > c$, and recalling the expression of~$b'$ and~$c'$  from Lemma~\ref{lem:align}, it follows from this lemma that
$$ \begin{array}{rcl}
   \max (\bar \xi_t (e), \bar \xi_t (e + 1)) & \n = \n & |\xi_{t-} (e - 1/2) - \xi_t (e + 1/2)| \vee |\xi_{t-} (e + 1/2) - \xi_t (e + 3/2)| \vspace*{4pt} \\
                                             & \n = \n & |a - b'| \vee |b' - c'| \geq \rho_- |a - b| = \rho_- \,\bar \xi_{t-} (e). \end{array} $$
 In all four cases, the implication~\eqref{eq:control-gap} holds, which proves the lemma.
\end{proof} \\ \\
 It follows from Lemma~\ref{lem:control-gap} and the discussion preceding Lemma~\ref{lem:D} that, as long as the gap at~$\zeta_t (e)$ at time~$t$ exceeds~$D$, this gap dominates stochastically the process~$X_t$ defined in~\eqref{eq:submartingale}.
 To complete the proof of the theorem, the last step is to study~$X_t$, showing in particular that, with positive probability when starting above~$2D$, the process converges to infinity without ever going below~$D$, in which case the stochastic domination would fail.
\begin{lemma}
\label{lem:supermartingale}
 There exists~$0 < c_0 < 1$ such that~$Y_t = c_0^{\log (X_t)}$ is a supermartingale.
\end{lemma}
\begin{proof}
 Define~$\phi (c) = 2c^{\log (\rho_-)} + c^{\log (\rho_+)} - 3$.
 From
 $$ \begin{array}{rcl}
    \rho_-^2 \rho_+ & \n = \n &
    \displaystyle \bigg(\frac{1 + 2 \mu_+}{1 + 3 \mu_+} \bigg)^2 (1 + 2 \mu_+) =
    \displaystyle \frac{(1 + 2 \mu_+)^3}{(1 + 3 \mu_+)^2} =
    \displaystyle \frac{1 + 6 \mu_+ + 12 \mu_+^2 + 8 \mu_+^3}{1 + 6 \mu_+ + 9 \mu_+^2} > 1, \end{array} $$
 it follows that~$\phi$ is increasing at point one:
 $$ \phi' (1) = 2 \log (\rho_-) + \log (\rho_+) = \log (\rho_-^2 \rho_+) > \log (1) = 0. $$
 Observing also that~$\phi (1) = 0$, we deduce that
 $$ \phi (c_0) < 0 \quad \hbox{for some} \quad 0 < c_0 < 1 $$
 fixed from now on.
 In particular,
 $$ \begin{array}{rcl}
    \lim_{\ep \to 0} \ep^{-1} E (Y_{t + \ep} - Y_t \,| \,X_t) & \n = \n &
      2 (c_0^{\log (X_t) + \log (\rho_-)} - c_0^{\log (X_t)}) + (c_0^{\log (X_t) + \log (\rho_+)} - c_0^{\log (X_t)}) \vspace*{4pt} \\ & \n = \n &
      Y_t \,(2 c_0^{\log (\rho_-)} + c_0^{\log (\rho_+)} - 3) = Y_t \,\phi (c_0) \leq 0, \end{array} $$
 showing that the process~$Y_t$ is a supermartingale.
\end{proof}
\begin{lemma}
\label{lem:OST}
 Assume that~$X_0 > 2D$. Then,
 $$ \begin{array}{l} P (\lim_{t \to \infty} X_t = \infty \ \hbox{and} \ X_t > D \ \hbox{for all} \ t) \geq 1 - c_0^{\log (2)} > 0. \end{array} $$
\end{lemma}
\begin{proof}
 For all~$N > 2D$, define the stopping time
 $$ T_N = \inf \{t : X_t \notin (D, N) \}. $$
 By the optional stopping theorem applied to the process~$Y_t$ stopped at time~$T_N$,
\begin{equation}
\label{eq:OST1} 
 E (Y_{T_N}) \leq E (Y_0) = E (c_0^{\log (X_0)}) \leq c_0^{\log (2D)}.
\end{equation}
 Observe also that we have the lower bound
\begin{equation}
\label{eq:OST2}
\begin{array}{rcl}
  E (Y_{T_N}) & \n = \n & E (Y_{T_N} \,| \,X_{T_N} \leq D) P (X_{T_N} \leq D) + E (Y_{T_N} \,| \,X_{T_N} \geq N) P (X_{T_N} \geq N) \vspace*{4pt} \\
              & \n \geq \n & E (Y_{T_N} \,| \,X_{T_N} \leq D) (1 - P (X_{T_N} \geq N)) \vspace*{4pt} \\
              & \n \geq \n & c_0^{\log (D)} (1 - P (X_{T_N} \geq N)). \end{array}
\end{equation}
 Combining~\eqref{eq:OST1} and~\eqref{eq:OST2}, we conclude that
 $$ P (X_{T_N} \geq N) \geq 1 - c_0^{\log (2D)} / c_0^{\log (D)} = 1 - c_0^{\log (2D) - \log (D)} = 1 - c_0^{\log (2)} > 0. $$ 
 Because the inequality holds for all~$N > 2D$, the lemma follows.
\end{proof} \\ \\
 Using Lemmas~\ref{lem:increase-gap} and~\ref{lem:OST}, and that, according to Lemma~\ref{lem:control-gap}, the gap at~$\zeta_t (e)$ at time~$t$ dominates stochastically the process~$X_t$ as long as it exceeds~$D$, we conclude that, for all edges~$e$,
 $$ \begin{array}{rcl}
 P (\bar \xi_t (\zeta_t (e)) \to \infty) & \n \geq \n &
 P (\bar \xi_t (\zeta_t (e)) \to \infty \,| \,\bar \xi_1 (\zeta_1 (e)) > 2D) \,P (\bar \xi_1 (\zeta_1 (e)) > 2D) \vspace*{4pt} \\ & \n \geq \n &
 P (X_t \to \infty \ \hbox{and} \ X_t > D \ \hbox{for all} \ t \,| \,X_0 > 2D) \,P (\bar \xi_1 (\zeta_1 (e)) > 2D) \vspace*{4pt} \\ & \n \geq \n &
 (1/2 - \theta / 4)^2 (\theta/ 4)^2 (e^{-5} / K!)(1 - c_0^{\log (2)}) > 0. \end{array} $$
 This shows that~\eqref{eq:divergence} holds.
 As previously explained, this and the ergodic theorem imply the existence of infinitely many edges~$e$ such that~$\bar \xi_t (\zeta_t (e)) \to \infty$, which proves the theorem.


\bibliographystyle{plain}
\bibliography{biblio.bib}

\end{document}